\documentclass[11pt,twoside, a4paper, english, reqno]{amsart}
\usepackage[dvips]{epsfig}
\usepackage{amscd}
\usepackage{amssymb}
\usepackage{amsthm}
\usepackage{amsmath}
\usepackage{latexsym}
\usepackage{amsaddr}

\usepackage{upref}
\usepackage{hyperref}

\hypersetup{linkcolor=blue, colorlinks=true
,citecolor = red}
\usepackage{color}
\theoremstyle{plain}
\newtheorem{thm}{Theorem}[section]

\theoremstyle{plain}
\newtheorem{lem}[thm]{Lemma}

\theoremstyle{definition}
\newtheorem{defi}{Definition}[section]

\newcommand{\bn}{\mathbb{B}^{N}}
\newcommand{\rn}{\mathbb{R}^{N}}

\newcommand{\rnn}{\mathbb{R}^{2}}
\newcommand{\hn}{\mathbb{H}^{N}}
\newcommand{\hnn}{\mathbb{H}^{2}}

\newcommand{\Chn}{C^{\infty}_c(\mathbb{H}^N)}
\newcommand{\authorfootnotes}{\renewcommand\thefootnote{\@fnsymbol\c@footnote}}%

\numberwithin{equation}{section} \allowdisplaybreaks
\begin{document}
\title{Adams Inequality on The Hyperbolic Space}
\author{Debabrata Karmakar and Kunnath Sandeep $^\dagger$}
\thanks{$^\dagger$ TIFR  Centre for Applicable Mathematics, Post Bag No. 6503
 Sharadanagar,\\Yelahanka New Town, Bangalore 560065.\\ Emails: debkar@math.tifrbng.res.in , sandeep@math.tifrbng.res.in}

\begin{abstract}
In this article we establish the following Adams type inequality in the Hyperbolic space $\hn$:
$$\sup_{u \in C^{\infty}_c(\hn),\int\limits_{\hn}(P_ku)u \ dv_g\leq 1} \int_{\hn} (e^{\beta u^2} - 1) \ dv_g < \infty $$
iff $\beta \leq \beta_0(N,k)$ where, $2k=N$, $P_k$ is the critical GJMS operator in $\hn$ and $\beta_0(N,k)$ is as defined in \eqref{beta0}. As an application we prove the asymptotic behaviour of the best constants in Sobolev inequalities when $2k=N$ and also prove some existence results for the $Q_k$ curvature type equation in $\hn$. 
\end{abstract}

\maketitle
MSC2010 Classification: \em{46E35, 26D10}\\
Keywords: \em{Adams inequality, Hyperbolic space}
\section{Introduction} The main focus of this article is on the optimal Adams inequality in space forms. This inequality was established in the zero curvature case $\mathbb{R}^N$ by D.R. Adams(\cite{A}) and in the constant positive sectional curvature case by Fontana (\cite{Fonta}). In this article we establish it in the case of Hyperbolic space. The inequality we prove (see Theorem \ref{HYADA}) is in view of the PDE which governs the critical $Q_{\frac N2}$ curvature under a conformal change of the metric.

Recall the Sobolev embedding theorem which states that if $\Omega$ is a bounded domain in $\mathbb{R}^N,$ then the Sobolev space $H^{k}_0(\Omega)$ is continuously embedded in $L^p(\Omega)$ for all $1\le p \le \frac{2N}{N-2k},$ if $2k <N$ and when $2k >N$, $H^{k}_0(\Omega)$ is continuously embedded in $C^{m,\alpha}(\Omega)$ where $m=k-\left[\frac N2\right]-1 $ and $\alpha =\left[\frac N2\right]+1 -\frac{N}{2}$ if $N$ is odd, otherwise $\alpha \in (0,1)$ is any arbitrary number. One can easily see that when $N=2k$, neither of the above embeddings are true.\\\\
 When $k=1$, an embedding for this case was obtained by  Poho\v{z}aev (\cite{Pohozaev}) and Trudinger(\cite{Trudinger}).\\
It is well known that the optimal Sobolev embedding play an important role in several geometric pdes, like the Yamabe problem, Prescribing the scalar curvature, etc. In 1971 J.Moser (\cite{Moser}) while trying to study the question of prescribing the Gaussian curvature on the sphere understood the need for establishing a sharp form of the embedding obtained by Poho\v{z}aev and Trudinger. He showed that there exists a positive constant $C_0$ depending only on $N$ such that
\begin{align} \label{Moser}
\sup_{u \in C^{\infty}_c(\Omega), \int_{\Omega} |\nabla u|^N \leq 1} 
\int_{\Omega} e^{\alpha |u|^{\frac{N}{N - 1}}} \ dx \leq C_0 |\Omega|,
\end{align}
holds for all $\alpha \le \alpha_N=N[\omega_{N-1}]^{\frac{1}{N-1}}$, where $\Omega$  is a bounded domain in $\rn,$ and $|\Omega|$ denotes the volume 
of $\Omega$ and $\omega_{N-1}$ is the $N-1$ dimensional measure of $S^{N-1}.$ Moreover when $\alpha > \alpha_N$, the above supremum is infinite.\\\\
In $1988,$ D.R. Adams(\cite{A}) established the sharp embedding in the case of higher order Sobolev spaces. He found the sharp constant $\beta_0$ for the higher order Trudinger-Moser type inequality. More precisely he proved that if $k$ is a positive integer less than $N,$ then there exists a constant $c_0 = c_0(k,N)$ such that 
\begin{align} \label{Adamsintegral}
 \sup_{u \in C^{k}_c(\Omega), \int_{\Omega} |\nabla^k u|^p \leq 1}\int_{\Omega} e^{\beta |u(x)|^{p^{\prime}}} \ dx \leq c_0 |\Omega|,
\end{align}
for all $\beta \leq \beta_0(k,N),$ where $p = \frac{N}{k},p^{\prime} = \frac{p}{p - 1} $,
\begin{align} \label{beta0}
 \beta_0(k,N) = 
 \begin{cases}
  \frac{N}{\omega_{N-1}} \left[\frac{\pi^{\frac{N}{2}}2^k \Gamma\left(\frac{k + 1}{2}\right)}
  {\Gamma \left(\frac{N - k + 1}{2}\right)}\right]^{p^{\prime}}, \ \ \ \mbox{if $k$ is odd}, \\ \ \
  \frac{N}{\omega_{N-1}} \left[\frac{\pi^{\frac{N}{2}}2^k \Gamma\left(\frac{k}{2}\right)}
  {\Gamma \left(\frac{N - k }{2}\right)}\right]^{p^{\prime} }, \ \ \ \mbox{if $k$ is even}, \\
  \end{cases}
\end{align}
and $ \nabla^k $ is defined by
 \begin{align}\label{higra}
  \nabla^k :=
  \begin{cases}
   \Delta^{\frac{k}{2}} , \ \ \ \ \ \ \ \mbox{if} \ k \ \mbox{is even}, \\
   \nabla \Delta^{\frac{k-1}{2}} , \ \ \mbox{if} \ k \ \mbox{is odd}.
  \end{cases}
\end{align}
Furthermore, if $\beta > \beta_0,$ then the supremum in \eqref{Adamsintegral} is infinite.\\\\
There has been many extensions and improvements of these two inequalities.
Extensions of \eqref{Moser} to functions belonging to $W^{1,N}(\rn)$ were obtained by various authors, see  Cao(\cite{Cao}), Panda(\cite{Panda}), J.M. do \'{O}(\cite{DoO}), Ruf(\cite{Ruf}), Li-Ruf(\cite{LiRuf}) and the references therein. Extension of the same to the case of $\Omega$ with infinite measure has been dealt with in \cite{MS} and \cite{BM}. A significant improvement of \eqref{Moser} was obtained by Adimurthi-Druet(\cite{AdiDruet}) in dimension $N=2$ and was extended to higher dimensions by Yang(\cite{Yang}). See also \cite{AdiSandeep}, \cite{Ode},\cite{AdiYang} for various other improvements. Like wise \eqref{Adamsintegral} has also attracted various generalizations and improvements. See Tarsi (\cite{tar}) for details.\\\\
In \cite{Moser} Moser also proved a sharp version of \eqref{Moser} on $S^2$, and subsequently  Fontana in (\cite{Fonta}) obtained the following  sharp version of \eqref{Adamsintegral} on compact Riemannian manifolds :\\
Let $(M,g)$ be an $N$ dimensional compact Riemannian manifold without boundary, and $k$  a positive integer less than $N,$ then there exists a constant $c_0 = c_0(k,M)$ such that
\begin{align} \label{Fontanaintegral}
 \sup_{u \in C^{k}(M),\int_Mu =0, \int_{M} |\nabla^k_g u|^p \leq 1}\int_{M} e^{\beta |u(x)|^{p^{\prime}}} \ dv_g \leq c_0,
\end{align}
if $\beta \leq \beta_0(k,N),$ where $p,p^{\prime}$ are as above and $\nabla_g^k $ is defined as in \eqref{higra} with $\nabla$ and $\Delta$  the gradient and Laplace Beltrami operators with respect to the metric $g.$ Furthermore, if $\beta > \beta_0,$ then the supremum in \eqref{Fontanaintegral} is infinite.\\\\
Our aim in this article is to study the Adams type inequality in the Hyperbolic space $\hn.$ More precisely we study the optimal embeddings of the Sobolev space $H^{k}(\hn)$ when $k$ is a positive
integer and $N = 2k.$ One of the main difficulty one faces in the full Hyperbolic space is due to its infinite measure or equivalently in coordinates the Hardy type singularity present in the integrals.\\\\
For $k = 1, N = 2,$ Mancini-Sandeep(\cite{MS}) proved the Trudinger-Moser inequality
in the hyperbolic space or in other words $H^{1}(\hnn)$ is embedded into the Zygmund space $Z_{\phi}$ determined by 
the function $ \phi = (e^{4\pi u^2} - 1).$ Another proof of this inequality was given by  Adimurthi-Tinterev(\cite{AdiT}). In fact in \cite{MS}, they obtained the following general theorem:\\
 Let $\mathbb{D}$ be the unit open disc in $\rnn,$ endowed with a conformal metric
$h = \rho g_e,$  where $g_e$ denotes the Euclidean metric and $\rho \in C^2(\mathbb{D}), \rho > 0,$ then
\begin{align} \label{TMH2}
 \sup_{u \in C^{\infty}_c(\mathbb{D}), \int_{\mathbb{D}} |\nabla_h u|^2 \leq 1} \int_{\mathbb{D}} 
                     \left(e^{4\pi u^2} - 1\right) \ dv_h < \infty,
\end{align}
holds true if and only if $h \leq c g_{\hnn}$ for some positive constant $c.$ Here $\nabla_h, dv_h$ denotes respectively the gradient and volume element for the metric $h$ and 
$g_{\hnn} = \displaystyle{\sum_{i = 1}^2} \left(\frac{2}{1 - |x|^2}\right)^2 dx^2_i$ is the Poincare metric in the disc. 

Extensions of this inequality to $N>2$ were obtained by Lu-Tang (\cite{LuT}) and \cite{BM}. 
The study of Trudinger-Moser type inequality on the hyperbolic space $\hn(N \geq 2)$ has been investigated by various 
authors for the past few years. For works related to sharp  Trudinger-Moser type inequality on the hyperbolic space
we refer to \cite{LuT}, \cite{MS}, \cite{MST}, \cite{Tin}, \cite{YZ} and the references therein. \\\\
One can think of Adams inequality in the hyperbolic space in various ways. However recall that the original motivation of Moser in establishing the sharp Moser-Trudinger inequality was to solve the question of prescribing the Gaussian curvature on the sphere $S^N$ by changing the metric conformally. In the same spirit one can consider the question of prescribing the optimal $Q_{\frac N2}$ curvature. More precisely if $(M,g)$
a Riemannian manifold of even dimension $N =2k$ with the  $Q_{\frac{N}{2}}$ curvature $Q_k$, let $\tilde g =e^{2u}g$ be a conformal metric on $M$, then the  $Q_{\frac{N}{2}}$ curvature $\tilde Q_k$ of $\tilde g$ and that of $g$ are related by the formula,
$$P_k(u) + Q_k = \tilde Q_k e^{Nu}, $$
where $k=\frac N2$ and $P_k$ is the critical GJMS operator on $(M,g)$. In view of this PDE and considering its variational structure, the right Adams inequality one should explore is the exponential integrability of $C_c^k$ functions with the constraint $\int\limits_{M}P_k(u)u\ dv_g \le 1$.    
We establish such an embedding in the case of Hyperbolic space. The main result of this article is the following:\\

\begin{thm}\label{HYADA}
Let $\hn$ be the $N$ dimensional hyperbolic space with $N$ even and $k =\frac N2$ then,
 \begin{align}
  \sup_{u \in \Chn , ||u||_{k,g} \leq 1} \ \int_{\hn} \left(e^{\beta u^2} - 1\right) \ dv_g < +\infty
 \end{align}
iff $\beta \leq \beta_0(k,N)$, where $\beta_0(k,N)$ is as defined in \eqref{beta0} and $||u||_{k,g}$ is the norm defined by
\begin{align} 
 ||u||_{k,g} := \left[\int_{\hn} (P_k u)u \ dv_g\right]^{\frac 12}, \ \mbox{for all} \ \ u \in \Chn,
\end{align}
where $P_k$ is the $2k$-th order GJMS operator on the hyperbolic space $\hn$.
\end{thm}
See Section 3 for more details about GJMS operator and the above norm.\\\\ From the above inequality we will also derive the exact asymptotic behaviour of the best constants in the Sobolev inequality when $N=2k$, see Theorem \ref{estimate on Spk} for the precise statement. Also in Section 5.2, as an application of the above inequality we will discuss the solvability of the PDE which governs the $Q_{\frac N2}$ curvature in the hyperbolic space in its variational setting.\\

After this work was completed we came to know about the preprint \cite{FoM}, where an Adams inequality is established in the Hyperbolic space $\hn$ for all $N.$ Our inequality is different from the one established in \cite{FoM} and the proofs are different. Also when $N=4,6$ and $8$ we can show that our inequality is stronger than the one in \cite{FoM} and we believe it is true for all even $N$.

\section{Notations and Preliminaries}
\subsection{Notations}
 
For a bounded domain $\Omega$ in $\rn,$ we will denote by $H^{k}(\Omega),$ the usual Sobolev space, with respect to the norm,
\begin{align*}
 ||u||_{H^{k}(\Omega)} := \left(\sum_{|\alpha| \leq k} ||D^{\alpha} u||^2_{L^2(\Omega)}\right)^{\frac{1}{2}},
\end{align*}
where $\alpha$ is a multi-index, $\alpha = (\alpha_1, \alpha_2, ..., \alpha_N),\;\alpha_i \in \mathbb{N} \cup \{0\},$
$|\alpha| = \alpha_1 + \alpha_2 + ... + \alpha_N,$
\begin{align*}
 D^{\alpha}  := \frac{\partial^{|\alpha|}}{\partial x^{\alpha_1}_1 \partial x^{\alpha_2}_2 ... \partial x^{\alpha_N}_N} ,
\end{align*}
We will denote by $H^{k}_0(\Omega),$ the closure of $C^{\infty}_c(\Omega)$ in $H^{k}(\Omega).$ There are a few equivalent norms in $H^{k}_0(\Omega),$ we will collect a few of them in the next lemma, whose proofs are well known.
\begin{lem}\label{equieucli} 
Let $\Omega$ be a bounded open set in $\rn$ define $||u||_{k, \Omega}$ and $|||u|||_{k, \Omega}$ as
\begin{align}
 ||u||_{k, \Omega} := \left(\sum_{l = 0}^k ||\nabla^l u||^2_{L^2(\Omega)}\right)^{\frac{1}{2}},\; u \in H^{k}(\Omega)
\end{align}
\begin{align}
 |||u|||_{k, \Omega} := ||\nabla^k u||_{L^2(\Omega)}\; , u \in H^{k}(\Omega)
\end{align}
then $||u||_{k, \Omega}$ and $|||u|||_{k, \Omega}$ are equivalent norms in 
$H^{k}_0(\Omega)$. 
\end{lem}
We will be using the following boundary Hardy-Rellich inequality for the polyharmonic operator established by M. Owen (see \cite{Owen}).
\begin{lem} \label{Owen}
 Let $\Omega$ be a bounded convex domain in $\rn,$ and $d(x) := d(x, \partial \Omega)$ be the distance from
 $x$ to the boundary of $\Omega$ then,
 \begin{align}
  A(k) \int_{\Omega} \frac{u^2}{d^{2k}(x)} \ dx \le  |||u|||_{k, \Omega}^2, \ \ 
  \mbox{for all} \ u \in C^{\infty}_{c}(\Omega),
 \end{align}
 where $A(k) = \frac{1^2 . 3^2 ... (2k - 1)^2}{4^k}$ and it is sharp.
\end{lem}
{\bf Hyperbolic space :}
 The hyperbolic $N$-space is a complete, simply connected, noncompact Riemannian $N$-manifold having constant section
 curvature equals to $-1,$ and any two manifolds sharing above properties are isometric(see \cite{Wolf}).
We will denote the hyperbolic $N$-space by $\hn.$ \\\\
There are several models for the hyperbolic $N$-space
$\hn,$ commonly used are the ball model, the half space model, the Lorentz model. In this paper we will be using the ball model $(\bn, g_{\hn})$ where
$\bn := \{x = (x_1, x_2, ... ,x_N) \in \rn : (x^2_1 + x^2_2 + ... + x^2_N) < 1\}$ and $g_{\hn}$ is the Poincare metric given by
\begin{align} \label{metric}
 g_{\hn} = \sum_{i = 1}^N \left(\frac{2}{1 - |x|^2}\right)^2 dx^2_i,
\end{align}
From now on $\hn$ will stands for the conformal ball model, and we will simply write $g$ instead of $g_{\hn}$ to denote the metric
on $\hn.$\\\\
The volume element for $\hn$ is given by $dv_g = \left(\frac{2}{1 - |x|^2}\right)^N dx,$ where $dx$ denotes the Lebesgue measure on $\rn.$\\
Let $\nabla_g$ and $\Delta_g$ denotes respectively the hyperbolic gradient 
and Laplace-Beltrami operator, then in terms of local coordinates $\nabla_g$ and $\Delta_g$ takes the form :
\begin{align} \label{gradlaplacian}
 \nabla_g = \left(\frac{1 - |x|^2}{2}\right)^2\nabla \ , \ 
 \Delta_g = \left(\frac{1 - |x|^2}{2}\right)^2 \Delta + (N - 2)\left(\frac{1 - |x|^2}{2}\right) \langle x,\nabla \rangle,
\end{align}
where $\nabla, \Delta$ are the usual Euclidean gradient and Laplacian respectively, and $\langle.,.\rangle$ is the 
standard inner product in $\rn.$\\
Next we define the concept of Hyperbolic translation.
\begin{defi}[Hyperbolic Translation]
 For $b \in \bn $ we define the hyperbolic translation $\tau_{b}: \bn\rightarrow \bn$ by 
  \begin{align} \label{hyperbolictranslation}
  \tau_{b}(x) := \frac{(1 - |b|^2)x + (|x|^2 + 2\langle x,b \rangle + 1)b}{|b|^2|x|^2 + 2\langle x,b \rangle + 1}.
 \end{align}
\end{defi}
 Then $\tau_{b}: \bn\rightarrow \bn$ is an isometry, see (see \cite{Ratcliffe}, theorem 4.4.6) for details and further discussions on isometries. As a consequence we immediately have :
  \begin{lem} \label{lemma1}
  Let $\tau_b$ be the hyperbolic translation of $\bn$ by $b.$ Then,
  \item[(i).] For all $u \in \Chn,$ there holds,
  \begin{align*}
   \Delta_g (u \circ \tau_b) = (\Delta_g u) \circ \tau_b, \ \ \langle \nabla_g (u \circ \tau_b), \nabla_g (u \circ \tau_b)\rangle_g = \langle(\nabla_g u) \circ \tau_b, (\nabla_g u) \circ \tau_b\rangle_g.
  \end{align*}
  \item[(ii).]For any $u\in \Chn$ and open subset $U$ of $\bn$
  \begin{align*}
  \int_{U} |u \circ \tau_b|^p \ dv_g = \int_{\tau_b(U)} |u|^p \ dv_g, \ \mbox{for all} \ 1 \leq p < \infty.
 \end{align*}
  
\end{lem}

\subsection{The Sobolev space $H^k(\hn)$:} For a positive integer $l,$ let $\Delta^l_g$  denotes the $l$-th iterated 
Laplace-Beltrami operator. Define,
\begin{align} \label{iteratedgradient}
 \nabla^l_g :=
 \begin{cases}
  \Delta^{\frac{l}{2}}_g , \ \ \ \ \ \ \mbox{ if $l$ is even} \\
  \nabla_g \Delta^{\frac{l - 1}{2}}_g, \ \ \mbox{if $l$ is odd}.
 \end{cases}
\end{align}
\begin{defi}
We define the space $H^k(\hn)$ as the completion of $C_c^\infty(\hn)$ with respect to the norm
\begin{align}
 ||u||_{H^k(\hn)} :=  \left[\sum^k _{m = 0}\int_{\hn} 
 |\nabla^m_g u|^2_g  \ dv_g\right]^{\frac{1}{2}},
\end{align}
where  $|\nabla^l_g u|_g $ is given by,
 \begin{align*}
  |\nabla^l_g u|_g :=
  \begin{cases}
   |\nabla^l_g u|, \ \ \ \ \ \ \ \ \ \ \ \mbox{if $l$ is even}, \\
   \langle \nabla^l_g u, \nabla^l_g u \rangle^{\frac{1}{2}}_g ,\ \ \mbox{if $l$ is odd}. 
  \end{cases}
\end{align*}
\end{defi}

In $H^k(\hn)$ we have the following higher order Poincare type inequalities:
\begin{lem}\label{higpoi}
Let $k,l$ be non-negative integers such that $l<k$, then the inequality
$$ \left(\frac{N-1}{2}\right)^{2(k-l)}\int_{\hn}|\nabla^l_g u |^2_g \ dv_g\; \le \int_{\hn}|\nabla^k_g u |^2_g \ dv_g $$
holds for all $u\in H^k(\hn)$. As a consequence
\begin{align}
 |||u|||_{H^{k}(\hn)} :=\left[\int_{\hn}|\nabla^k_g u |^2_g \ dv_g\right]^{\frac{1}{2}}\; , u\in H^k(\hn) 
\end{align}
defines an equivalent norm in $H^k(\hn)$.
\end{lem}
\begin{proof} We know from the Poincare inequality that
\begin{equation}\label{poi}
\left(\frac{N-1}{2}\right)^2\int_{\hn}|u |^2 \ dv_g\; \le \int_{\hn}|\nabla_g u |^2_g \ dv_g.
\end{equation}  
holds for all $u\in C_c^\infty(\hn)$. Now  
\begin{equation}\label{poi2}
\int_{\hn}|\nabla_g u |^2_g \ dv_g = \int_{\hn}(-\Delta_g u )u \ dv_g \le \left[\int_{\hn}(\Delta_g u )^2 \ dv_g\right]^{\frac 12} \left[\int_{\hn} u^2 \ dv_g \right]^{\frac 12}.
\end{equation} 
Combining this with the \eqref{poi} inequality gives gives 
$$ \left(\frac{N-1}{2}\right)^2\int_{\hn}|\nabla_g u |^2_g \ dv_g\; \le \int_{\hn}|\nabla^2_g u |^2_g \ dv_g .$$
Assume by induction
$$\int_{\hn}|u |^2 \ dv_g\; \le  A\int_{\hn}|\nabla_g u |^2_g \ dv_g 
\le A^2\int_{\hn}|\nabla^2_g u |^2_g \ dv_g  \le... \le A^{k}\int_{\hn}|\nabla^k_g u |^2_g \ dv_g ,$$
where $A = \left(\frac{N-1}{2}\right)^{-2}$. We claim that the above inequality extends to $k+1$.\\
Suppose $k$ is even, then by using the inequality \eqref{poi} to $\Delta^{\frac k2}_g u$ we get
$$\int_{\hn}|\nabla^{k}_g u |^2_g \ dv_g =\int_{\hn}|\Delta^{\frac k2}_g u |^2_g \ dv_g \le A\int_{\hn}|\nabla_g \Delta^{\frac k2}_gu |^2_g \ dv_g  = A\int_{\hn}|\nabla^{k+1}_g u |^2_g \ dv_g.$$
When $k$ is odd applying \eqref{poi2} to $\Delta^{\frac{k-1}{2}}u$ we get
$$\int_{\hn}|\nabla^{k}_g u |^2_g \ dv_g =\int_{\hn}|\nabla_g
\Delta^{\frac{k-1}{2}}_g u |^2_g \ dv_g \le A\int_{\hn}|\Delta_g \Delta^{\frac{k-1}{2}}_gu |^2_g \ dv_g  = A\int_{\hn}|\nabla^{k+1}_g u |^2_g \ dv_g.$$
This completes the induction argument and hence the lemma follows.
\end{proof}
\section{GJMS operator and a conformally equivalent norm } Let $(M,g)$ be a Riemannian manifold of dimension $N$. We know that the conformal Laplacian or the Yamabe operator $P_{1,g}$ defined by $$P_{1,g}= -\Delta_g +\frac{N-2}{4(N-1)}R_g, $$ where $R_g$ is the scalar curvature of the metric, is a conformally invariant differential operator in the sense that if $\tilde g = e^{2u}g$ is a conformal metric then  $$P_{1,\tilde g}(v) = e^{-(\frac N2 +1)u}P_{1,g}(e^{(\frac N2 -1)u} v),$$ for all smooth functions $v.$ A fourth order conformally invariant operator with leading term $\Delta_g^2$ was invented by Paneitz and later Branson found a conformal sixth order operator with leading term  $\Delta_g^3$. Existence of a general conformal operator of higher degree was obtained by Graham, Jenne, Mason and Sparling (\cite{GJMS}) what is popularly known as GJMS operators. 
It follows from their work that when $(M,g)$ is a Riemannian manifold of even dimension $N$ then for $k\in \{1,2,...,\frac N2\}$ there exists a conformally invariant differential operator $P_{k,g}$ of the form 
$P_{k,g} = \Delta_g^k +  lower\;order\; terms$, satisfying for a conformal metric  $\tilde g = e^{2u}g$, 
\begin{equation}\label{confrel}
P_{k,\tilde g}(v) = e^{-(\frac N2 +k)u}P_{k,g}(e^{(\frac N2 -k)u} v).
\end{equation}
When $N$ is even and $k>\frac N2,$ a conformally invariant operator $P_{k,g}$ with the above properties may not exist in general. For this reason $P_{\frac N2, g}$ is known as the critical GJMS operator.\\\\
We are going to use this operators to define a conformally invariant norm 
in the space $H^k(\hn)$ for $1\le k\le \frac N2.$\\
For the simplicity of notation we will denote the GJMS operator $P_{k,g}$ in the hyperbolic space by $P_k$. It is known that (see \cite{Liu},\cite{juhl}) $P_{k}$ has an explicit expression given by:
\begin{align} \label{GJMS}
 P_k := P_1 (P_1 + 2) (P_1 + 6)...(P_1 + k(k - 1)).
\end{align}
After expanding we may write \eqref{GJMS} as 
\begin{align} \label{expression for GJMS operator}
 P_k = (-1)^k \left[\Delta^k_g + \sum^{k - 1}_{m = 0}  a_{km} \Delta^m_g\right] ,
\end{align}
where $a_{km}$ are non-negative constants.\\
One can easily verify the following lemma :
\begin{lem} \label{properties of P_k}
Let $\tau$ be an isometry of $\ \hn$ and $U$ be an open subset of $\ \hn,$ and $u \in \Chn,$ then  
 \item[(i).] \begin{align*} P_k(u \circ \tau) =  P_k(u)\circ \tau
 \end{align*}
  \item[(ii).] 
 \begin{align*}
  \int_{U} P_k(u \circ \tau) (u \circ \tau) \ dv_g = \int_{\tau(U)} (P_k u) u \ dv_g. \\
 \end{align*}
\end{lem}
In the next lemma we will define a conformally invariant norm:\\
\begin{lem} Let $||u||_{k,g}$ be defined by
\begin{align} \label{norm}
 ||u||_{k,g} := \left[\int_{\hn} (P_k u)u \ dv_g\right]^{\frac 12}, \  \ u \in \Chn,
\end{align}
then $||.||_{k,g}$ defines a norm on $\Chn$. When $N=2k$ there exists a positive constant $\Theta$ such that,
\begin{align}\label{equivalent}
  \frac{1}{\Theta} ||u||_{k,g} \leq ||u||_{H^k(\hn)} \leq \Theta ||u_{k,g},
 \end{align}
for all $u \in \Chn.$ 
\end{lem} 
\begin{proof} First observe that the hyperbolic space is obtained from the 
Euclidean unit ball by changing the metric conformally as $e^{2\phi} g_e$ where $g_e$ is the Euclidean metric and $\phi = \log\left(\frac{2}{1-|x|^2} \right)$. Thus using the conformal relation \eqref{confrel} and using the fact that the $P_{k,g_e} = \Delta ^k$ we get
\begin{equation}
\int_{\hn} (P_k u)u \ dv_g =\int_{\bn}(\Delta^kv) v \; dx = \int_{\bn}|\nabla^kv|^2 \; dx
\end{equation}
where 
\begin{equation}\label{confchan}
v(x)= \left(\frac{2}{1-|x|^2} \right)^{(\frac{N}{2}-k)}u(x).
\end{equation}
From this relation one can easily see that \eqref{norm} defines a norm.\\
To prove the equivalence of norms when $N=2k$, first note that
 \begin{align*}
  ||u||_{k,g} \leq \left(\max_{m} \sqrt{a_{km}} \right) \ ||u||_{H^k(\hn)}.
 \end{align*}
To prove the reverse inequality, first observe that we have $v =u$ in \eqref{confchan} in this case and consequently
\begin{align}
 \int_{\hn} (P_k u)u \ dv_g = \int_{\bn} |\nabla^k u|^2 \ dx.
\end{align}
So, using Lemma \ref{higpoi} it is enough to show that, $\int_{\hn} |\nabla_g^k u|^2_g \ dv_g$ can be estimated by $\int_{\bn} |\nabla^k u|^2 \ dx.$
We can show by induction that,
\begin{align*}
 |\nabla^k_g u|^2_g \leq C\left[\left(\frac{1 - |x|^2}{2}\right)^{2k} |\nabla^k u|^2 
 + \sum_{|\alpha| \leq k - 1} \left(\frac{1 - |x|^2}{2}\right)^{2|\alpha|} |D^{\alpha} u|^2 \right],
\end{align*}
where $C$ is a positive constant independent of $u.$
Integrating this relation against the hyperbolic measure and using Lemma \ref{Owen} and Lemma \ref{equieucli}, we get
\begin{align} \label{equivalence of norms inequality 1}
 \int_{\hn} |\nabla_g^k u|^2_g \ dv_g &\leq C \left[\int_{\bn} |\nabla^{k}u|^2 \ dx
 + \sum_{|\alpha| \leq k - 1} \int_{\bn} \frac{|D^{\alpha} u|^2}{(1 - |x|)^{N - 2|\alpha|}} \ dx\right], \notag \\
 &\leq C\left[\int_{\bn} |\nabla^{k} u|^2 \ dx 
 + \sum_{|\alpha| \leq k - 1} |\nabla^{k - |\alpha|}(D^{\alpha} u)|^2 \ dx\right], \notag \\
 &\leq C ||u||^2_{H^{k}(\bn)} \leq C \int_{\bn} |\nabla^k u|^2 \ dx.
\end{align}
This completes the proof of the lemma.
\end{proof}
\section{Proof of Adams Inequality} We will prove the Adams inequality by proving a local inequality and then extend it to the entire space by a covering argument like in \cite{AdiT}. We need a few lemmas to implement this strategy and we will prove them in the next section. 
\subsection{Basic lemmas:} For an open set $U\subset \bn$ define
\begin{align*}
 ||u||_{H^k_g(U)} :=  \left[\sum_{m = 0}^{k} \int_{U} |\nabla^m_g u|^2_g  \ dv_g\right]^{\frac 12}\;, u \in C^{k}(\overline{U}) . 
\end{align*}
We need the following lemma which connects the above norm with that of the Euclidean Sobolev norm.
\begin{lem} \label{basicl2}
 Let $k$ be any positive integer, and $V,U$ be open sets such that $\overline{V} \subset U \subset \overline{U} \subset \bn$, then
 there exists a constant $C_0 > 0$ such that 
 \begin{align}
  ||u||_{H^{k}(V)} \leq C_0 ||u||_{H^k_g(U)}, \ \ \mbox{for all} \ u \in C^{k}(\overline{U}).
 \end{align}
\end{lem}

\begin{proof}
Let $ V_1$ be an open set such that $\overline V \subset V_1 \subset \overline  V_1 \subset U.$
In the proof we will denote any universal constant by $C,$ and $C$ may change in every step.
By induction one can show that, for any even positive integer $l,$
 \begin{align} \label{basiclemma1}
  \nabla^l_g u = \left(\frac{1 - |x|^2}{2}\right)^l \nabla^l u + \sum_{|\alpha| \leq l - 1} a_{\alpha}(x)D^{\alpha} u ,
 \end{align}
where $a_{\alpha}$'s are smooth functions in $\bn.$ 

Therefore taking $\nabla_g$ on both sides of \eqref{basiclemma1} we get,
\begin{align} \label{basiclemma1 l odd}
 \nabla^{l + 1}_g u &= \left(\frac{1 - |x|^2}{2}\right)^{l+2} \nabla^{l+1} u + b_l(x) \nabla^{l} u \notag \\
 & + \sum_{|\alpha| \leq l - 1} \left[\left(\frac{1 - |x|^2}{2}\right)^2a_{\alpha}(x)\nabla(D^{\alpha} u)
 + b_{\alpha}(x) D^{\alpha} u \right], 
\end{align}
where $b_{\alpha}, b_l$'s are smooth vector valued functions defined on $\bn.$\\
Using the basic inequalities,
\begin{align*}
 (a + b)^2 &\geq (1 - \delta) a^2 - (\frac{1}{\delta} - 1)b^2, \ \ a,b \in \mathbb{R}, \delta\in (0,1)\\
 \left(\sum_{i = 1}^m a_i\right)^2 &\leq m \left(\sum_{i = 1}^m a^2_i\right), \ \ a_i \in \mathbb{R}, 
 \ \mbox{for all} \  i = 1,...,m
\end{align*}
and a simple estimation using \eqref{basiclemma1}, \eqref{basiclemma1 l odd}, leads to
\begin{align} \label{1}
 \int_{V_1} |\nabla^l_g u|^2_g \geq C_1(1 - \delta) \int_{V_1} |\nabla^l u|^2 - C(\delta) \sum_{|\alpha| \leq l-1} 
 \int_{V_1} |D^{\alpha} u|^2, 
\end{align}
for all $1 \leq l \leq k$ (here we used the fact that $a_{\alpha}, b_{\alpha}, b_l$ are smooth and $\left(1 - |x|^2\right)$ is bounded below and above by a positive constants on $V_1$) .\\ 
Now fix $1 < l_0 \leq k,$ then summing over $l = 1,2,...,l_0,$ we get from \eqref{1}
\begin{align} \label{3}
 \sum_{l = 0}^{l_0} \int_{V_1} |\nabla^l u|^2 &\leq C \sum_{l = 0}^{l_0} \int_{V_1} |\nabla^l_g u|^2_g  
  + C ||u||^2_{H^{l_0 - 1}(V_1)}, \notag \\
  & \leq C||u||^2_{H^{l_0}_g(U)} + C ||u||^2_{H^{l_0 - 1}(V_1)}.
\end{align}
Thus we have 

\begin{align} \label{4}
||u||^2_{l_0,V_1} \leq C ||u||^2_{H^{l_0}_g(U)} + C ||u||^2_{H^{l_0-1}(V_1)}.
\end{align}
Now we claim that there exists a constant $C>0$ such that
\begin{align} \label{intreg}
||u||_{H^{l_0}(V)} \le C\left[||u||_{l_0,V_1}+ ||u||_{H^{l_0-1}(V_1)}\right].
\end{align}
This follows directly from the interior elliptic regularity (see \cite{ADN}) when $l_0$ is even. When
$l_0$ is odd we can again use the interior elliptic regularity as follows to get \eqref{intreg}. 
In fact, when $l_0$ is odd, let $\alpha = (\alpha_1, \alpha_2,...,\alpha_N)$ be a multi index such that 
$|\alpha| = l_0.$ Let us assume $\alpha_i \neq 0,$ then by $H^{l_0-1}$ regularity,
\begin{align*}
 ||D^{\alpha}u||_{L^2(V)} & \leq ||\frac{\partial u}{\partial x_i}||_{H^{l_0 -1}(V)} 
                          \leq C ||\Delta^{\frac{l_0 - 1}{2}}\left(\frac{\partial u}{\partial x_i}\right)||_{L^2(V_1)} 
                          + C||\frac{\partial u}{\partial x_i}||_{L^{2}(V_1)} \\
                          &\leq C||\nabla \Delta^{\frac{l_0 - 1}{2}}u||_{L^2(V_1)} + C||u||_{H^{l_0 - 1}(V_1)} \\
                          &\le C\left[||u||_{l_0,V_1}+ ||u||_{H^{l_0-1}(V_1)}\right]
\end{align*}
and hence \eqref{intreg} follows. Now using \eqref{intreg} in \eqref{4} and $l_0\le k$ we get
\begin{align} \label{iteration}
||u||^2_{H^{l_0}(V)} \leq C ||u||^2_{H^{k}_g(U)} + C ||u||^2_{H^{l_0-1}(V_1)},
\end{align}
where $V_1$ is such that $\overline{V} \subset V_1 \subset \overline{V_1} \subset U.$ 
 Now starting with $l_0=k$ in \eqref{iteration}, and an iteration argument gives,
\begin{align}
 ||u||^2_{H^{k}(V)} &\leq C||u||^2_{H^{k}_g(U)} + C||u||^2_{H^{2}(V_2)}, \notag \\
                    &\leq C ||u||^2_{H^{k}_g(U)} + C||u||^2_{L^2(U)}
                    \leq C ||u||^2_{H^{k}_g(U)},
\end{align}
where $\overline{V} \subset V_2 \subset \overline{V_2} \subset U,$ and this  completes the proof of the Lemma.  
\end{proof}
Next we state a covering lemma whose proof we omit as it follows very much like the corresponding covering lemma in \cite{AdiT}(Lemma 3.3 and Corollary 3.4).

\begin{lem} \label{basicl5}
 Let $U,V$ be any open sets in $\bn$ such that, $\overline{V} \subset U \subset \overline{U} \subset \bn,$ then there
 exists a countable collection $\{b_i\}^{\infty}_{i = 1}$ of elements in $\bn,$ and  a positive number $M \in \mathbb{N},$
 such that, 
\item[(i).] $\{\tau_{b_i}(V)\}^{\infty}_{i = 1}$ covers $\bn$ with multiplicity not exceeding $M,$
\item[(ii).] $\{\tau_{b_i}(U)\}^{\infty}_{i = 1}$  have multiplicity not exceeding $M.$
\end{lem}
\subsection{Local inequalities} In this section we will establish the uniform exponential integrability in compact subsets of $\hn$.
\begin{lem} \label{basicl3}
 Let $U,V$ be as in lemma \eqref{basicl2}, then there exists a number $q > 0,$ such that,
 \begin{align}
  \displaystyle{\sup_{u \in C^{\infty}(\overline{U}), ||u||_{H^k_g(U)} \leq 1}} \int_V \left( e^{qu^2} - 1 \right ) \ dx \leq C_1 <
  \infty.
 \end{align}
\end{lem}
\begin{proof} 
 Let $T $ be an extension operator from $H^{k}(V)$ to $H^{k}_0(\bn)$. Then by Lemma \ref{basicl2},
 \begin{align}
  ||T(u)||_{H^{k}(\bn)} \leq C ||u||_{H^{k}(V)} \leq C_0 ||u||_{H^k_g(U)}, \ \ \mbox{for all} \
  u \in C^{\infty}(\overline{U}), 
 \end{align}
where $C,C_0$ are positive constants. In other words, there exists a constant $C_0>0$ satisfying
 \begin{align} \label{operator}
  \int_{\bn} |\nabla^k T(u)|^2 \ dx \leq C_0^2 ||u||^2_{H^k_g(U)},
   \ \mbox{for all} \ u \in C^{\infty}(\overline{U}) ,C_0>0.
 \end{align}
Therefore  for all $u \in C^{\infty}(\overline{U})$ with $||u||_U \leq 1,$ 
we have from \eqref{operator}, 
\begin{align}
 \int_{\bn} |\nabla^k \left(\frac{1}{C_0} T(u)\right)|^2 \ dx \leq 1.
\end{align}

Let us take $q = \frac{\beta_0(k,N)}{C_0^2},$ then by Adams inequality (see \cite{A}) , 
\begin{align}
\int_V \left(e^{qu^2} - 1 \right) \ dx 
&\leq \int_{\bn} \left(e^{\beta_0(k,N) \left(\frac{1}{C_0} T(u)\right)^2} - 1\right) \ dx 
\leq C_1, 
\end{align}
where $C_1$ is independent of  $T(u),$ and this completes the proof of the lemma.
\end{proof}
We also need the following refinement of the above Lemma:
\begin{lem} \label{basicl4}
 Let $U,V,q$ be as in Lemma \ref{basicl3}, then there exists a positive constant $C_2 > 0,$ such that for all
 $u \in C^{\infty}(\overline{U}),$ with $||u||_{H^k_g(U)} < 1,$ satisfies,
 \begin{align}
  \int_V \left( e^{qu^2} - 1 \right ) \ dx \leq C_2
  \frac{||u||^2_{H^k_g(U)}}{1 - ||u||^2_{H^k_g(U)}}.
 \end{align}
\end{lem}

\begin{proof}
 Let $ u \in C^{\infty}(\overline{U})$ with $||u||_{H^k_g(U)} < 1,$ then applying Lemma \ref{basicl3} to $\frac{u}{||u||_{H^k_g(U)}},$
 we get, for all $l \geq 1,$
 \begin{align}
 \frac{1}{||u||^{2l}_{H^k_g(U)}} \int_V q^l \frac{ u^{2l}}{l!} &\leq
 \int_V \left(e^{q\frac{u^2}{||u||^2_{H^k_g(U)}}} - 1\right) \ dx  \leq C_1. \ \ \
\end{align}
This implies,
\begin{align} \label{5}
 \int_V q^l \frac{ u^{2l}}{l!} &\leq C_1 ||u||^{2l}_{H^k_g(U)}, \ \ \mbox{for all} \ l \geq 1.
\end{align}
Now summing over all $l \geq 1,$ and using $||u||_{H^k_g(U)} < 1,$ we get from \eqref{5},
\begin{align}
 \int_V \left( e^{qu^2} - 1 \right ) \ dx \leq C_2
  \frac{||u||^2_{H^k_g(U)}}{1 - ||u||^2_{H^k_g(U)}},
\end{align}
and proves the lemma.
\end{proof}

\subsection{Proof of Theorem \ref{HYADA} } Fix two open sets $V,U$ as in Lemma \ref{basicl2}. Then by Lemma \ref{basicl5}, there exists a countable collection $\{b_i\}^{\infty}_{i = 1}
\subset \bn$ and a positive integer $M_0,$ such that,
\begin{align*}
 \bn = \bigcup^{\infty}_{i = 1} \tau_{b_i}(V) = \bigcup^{\infty}_{i = 1} \tau_{b_i}(U), 
\end{align*}
and $\{\tau_{b_i}(U)\}^{\infty}_{i = 1}$ have multiplicity less than $M_0.$\\ 
Let $u \in C_c^\infty(\hn)$ be such that $||u||_{k,g} \leq 1.$ 
Let us define the set,
\begin{align}
 I_u := \left\{i \in \mathbb{N} : ||u \circ \tau_{b_i}||^2_{H^k_g(U)} \geq \frac{q}{2 \beta_0(k,N)} \right\},
\end{align}
where $q$ is defined as in Lemma \ref{basicl3}. Let $card(A)$ denotes the cardinality of a set $A.$ Then we 
claim that, \\
\textbf{claim:} $card(I_u) \leq \alpha_0,$ and $\alpha_0$ is independent of $u.$ \\
\emph{Proof of the claim :}
Let us denote by $U_i:=\tau_{b_i}(U) $ then, using the fact that the covering $\{U_i\}$ has multiplicity at most $M_0$ and \eqref{equivalent} we get
\begin{align*}
 \frac{q}{2 \beta_0(k,N)} card(I_u) &\leq \sum^{\infty}_{i = 1} ||u \circ \tau_{b_i}||^2_{H^k_g(U)}
                                   \leq \sum^{\infty}_{i = 1} ||u||^2_{H^k_g(U_i)}
                                    \leq M_0 ||u||^2_{H^{k}(\hn)} \le C ,
\end{align*}
where $C$ is independent of $u$, this proves the claim.\\\\
If $j \in \mathbb{N} \setminus I_u$, then 
$||\sqrt{\frac{\beta_0(k,N)}{q} }(u \circ \tau_{b_j})||_ {H^k_g(U)}< \frac{1}{2}.$ Applying Lemma \ref{basicl4} to
$v :=\sqrt{\frac{\beta_0(k,N)}{q} }(u \circ \tau_{b_j})$ we get,
\begin{align}
 \int_{\tau_{b_j}(V)} \left( e^{\beta_0(k,N)u^2} - 1 \right) \ dv_g & \leq
  \int_V \left( e^{\beta_0(k,N)(u \circ \tau_{b_j})^2} - 1 \right) \ dv_g, \notag \\
  &\leq C \int_V \left( e^{\beta_0(k,N)(u \circ \tau_{b_j})^2} - 1 \right) \ dx, \notag \\
  &\leq C \int_V \left(e^{qv^2} - 1 \right) \ dx, \notag \\
  & \leq C ||v||^2_{H^k_g(U)} \leq C ||u \circ \tau_{b_j}||^2_{H^k_g(U)} \leq C ||u||^2_{H^k_g(U_j)}.
\end{align}
Adding these relations we get we get,
\begin{align}
 \sum_{i \in \mathbb{N} \setminus I_u} \int_{\tau_{b_i}(V)} \left( e^{\beta_0(k,N)u^2} - 1\right) \ dv_g 
 & \leq C \sum_{i \in \mathbb{N} \setminus I_u} ||u||^2_{H^k_g(U_i)} \leq M_0 ||u||^2_{H^{k}(\hn)} \le C .  
\end{align}
Where $C$ is independent of $u.$
Now if $i \in I_u$ then,
\begin{align} \label{interiorestimate}
 \int_{\tau_{b_i}(V)} \left( e^{\beta_0(k,N)u^2} - 1\right) \ dv_g
 & = \int_{V} \left( e^{\beta_0(k,N)(u \circ \tau_{b_i})^2} - 1\right) \ dv_g, \notag \\
 &\leq C \int_{V} \left( e^{\beta_0(k,N)(u \circ \tau_{b_i})^2} - 1\right) \ dx, \notag \\
 &\leq C \int_{\bn} \left( e^{\beta_0(k,N)(u \circ \tau_{b_i})^2} - 1\right) \ dx. 
\end{align}
Now
\begin{align*}
 \int_{\bn} |\nabla^k (u \circ \tau_{b_i})|^2 \ dx = ||u \circ \tau_{b_i}||^2_{k,g} = ||u||^2_{k,g} \leq 1.
\end{align*}
Therefore using the Euclidean Adam's inequality \eqref{Adamsintegral} in  \eqref{interiorestimate}, we get,
\begin{align*}
 \int_{\tau_{b_i}(V)} \left( e^{\beta_0(k,N)u^2} - 1\right) \ dv_g \leq C, \ \ \mbox{for all} \ i \in I_u.
\end{align*}
Adding over such finitely many $i$'s we get,
\begin{align}
 \sum_{i \in I_u} \int_{\tau_{b_i}(V)} \left( e^{\beta_0(k,N)u^2} - 1\right) \ dv_g \leq C(\alpha_0 + 1).
\end{align}
Since $\{\tau_{b_i}(V)\}^{\infty}_{i = 1}$ covers $\bn,$ we get,
\begin{align*}
 \int_{\hn} \left(e^{\beta_0(k,N)u^2} - 1\right) \ dv_g \leq C,
\end{align*}
where $C$ is independent of $u.$ \\\\
To complete the proof we have to show that $\beta_0$ is optimal. For this purpose define for $m\in \mathbb{N}$,
\begin{align*}
  v_m = 
 \begin{cases}
  \sqrt{\frac{\log m}{2 M}} + \frac{1}{\sqrt{2M \log m}} \sum_{l = 1}^{k - 1} \frac{(1 - m|x|^2)^{l}}{l}
  , \ \ \ \mbox{if} \ 0 \leq |x| < \frac{1}{\sqrt{m}}, \\
  - \sqrt{\frac{2}{M \log m}} \log |x|, \ \ \ \qquad \qquad \qquad \ \   \ \mbox{if} \ \frac{1}{\sqrt{m}} \leq |x| < 1, \\
  \xi_m(x), \ \ \ \qquad \qquad \qquad \ \ \qquad \qquad \ \ \  \ \mbox{if} \ |x| > 1.
 \end{cases}
\end{align*}
where $M = \frac{(4 \pi)^k (k - 1)!}{2},$ and $\xi_m$'s are radial functions chosen so that,
\begin{align*}
 \xi_m \in C^{\infty}(\overline{B_2(0)}), \ \ \ \ \ \xi_m |_{\partial B_1(0)} = \xi_m |_{\partial B_2(0)} = 0.
\end{align*}
In addition we assume for $l = 1,2, ..., k-1,$
\begin{align*}
 \frac{\partial^l \xi_m}{\partial r^l}  |_{\partial B_1(0)}  = (- 1)^l (l - 1)! \sqrt{\frac{2}{M \log m}}, \ \ \ \
 \frac{\partial^l \xi_m}{\partial r^l}  |_{\partial B_2(0)}  = 0,
\end{align*}
and $ \xi_m, |\nabla^l \xi_m|, |\nabla^k \xi_m|$ are all $O\left(\frac{1}{\sqrt{\log m}}\right).$

By direct computations we can see that $v_m \in H^{k}_0(B_2(0))$  for all $m,$ and,
\begin{align*}
 \int_{B_2(0)} |\nabla^k v_m|^2 \ dx = 1 + O\left(\frac{1}{\log m}\right).
\end{align*}
as $m\rightarrow \infty$. See \cite{CZ} for details.\\
For our case we take $\tilde u_m (x) = v_m(2x),$ then it is easy to see 
that $\tilde u_m \in H^{k}(\hn),$ for all $m,$ and
\begin{align}\label{generalmoser}
 ||\tilde u_m||^2_{k,g} = 1 + O\left(\frac{1}{\log m}\right)
\end{align}
as $m\rightarrow \infty$.\\
Define $u_m = \frac{\tilde u_m}{||\tilde u_m||_{k,g}},$
and let $ \beta > \beta_0(k,N),$ then we have,
\begin{align} \label{blow up example}
 \int_{\hn} (e^{\beta u_m^2} - 1) \ dv_g &\geq \int\limits_{\{|x| < \frac{1}{\sqrt{m}}\}} 
 (C m^{\frac{\beta}{2M}} - 1) \ dx 
 \geq \frac{\omega_{N - 1}(C m^{\frac{\beta}{2M}} - 1) }{N m^k}\\
&\geq \frac{\omega_{N - 1}}{N} (C m^{\frac{\beta}{2M} - k} - m^{-k} )
\end{align}
It is easy to see that $\frac{\beta}{2M} > k,$ when $ \beta > \beta_0(k,N),$ and therefore the right hand 
side of \eqref{blow up example} tends to infinity as $m$ approaches to infinity. This completes the proof of the theorem.

 \section{Applications to PDE} In this section we will give two applications of the Adams inequality we proved. The first application will be the asymptotic estimates on the best constant in the Sobolev embedding when $N=2k$ and as a second application we will study certain PDEs in hyperbolic space motivated by the $Q_{\frac N2}$ curvature equation.
\subsection{Asymptotic estimates on best constants.} It is known from the work of G.Liu (see \cite{Liu}) that
 when $N > 2k,$  the Sobolev space  $H^{k}(\hn)$ is embedded in $L^q(\hn),$ where $q = \frac{2N}{N - 2k}.$ He proved the following sharp inequality :
 \begin{align} \label{Liu}
  \left(\int_{\hn} |u|^q \ dv_g \right)^{\frac{2}{q}} \leq \Lambda_k ||u||_{k,g}^2,\; u \in C^{\infty}_c(\hn)
 \end{align}
where $q = \frac{2N}{N - 2k}$ and $\Lambda_k$ is the best constant in this this inequality and is given by
\begin{align*}
 \Lambda_k = \frac{2^{2k} \omega^{-\frac{2k}{N}}_N}{N[N - 2k][N^2 - (2(k - 1))^2][N^2 - (2(k - 2))^2]...[N^2 - 2^2]} \ .
\end{align*}
When $N = 2k,$ clearly the exponent $q$ becomes infinity but one can easily see that $H^{k}(\hn)$ does not embeds in to $L^\infty.$  However it follows from the Adam's inequality (Theorem \ref{HYADA}) that the inequality 
 \begin{align}
  S_{k,p}\left[\int_{\hn} |u|^p \ dv_g \right]^{\frac{2}{p}} \leq ||u||_{k,g}^2 \; ,\, u \in H^{k}(\hn)
 \end{align}
 holds for all $p \geq 2$, with the best constant $S_{p,k}>0.$ Clearly $S_{p,k}\rightarrow 0$ as $p\rightarrow \infty$. We prove a precise asymptotic estimate for $S_{p,k}$ as $p$ goes to infinity.
 
 \begin{thm} \label{estimate on Spk}
  Let $k$ be a positive integer and $N = 2k.$ Then,
  \begin{align}
   S_{p,k} := \inf_{u \in \Chn, u \neq 0} \frac{||u||_{k,g}^2}{\left[\int_{\hn} |u|^p \ dv_g \right]^{\frac{2}{p}}} 
   = \frac{2 \beta_0(k,N) e + o(1)}{p},
  \end{align}
  as $p \rightarrow \infty.$
\end{thm}
\begin{proof}
 For simplicity of the notations we will write $\beta_0$ for $\beta_0(k,N).$ Let $u \in \Chn$ with
 $||u||_{k,g}\leq 1.$ Then by Adams inequality there exists a constant $C,$ independent of $u,$
 such that,
 \begin{align*}
  \int_{\hn} (e^{\beta_0 u^2} - 1) \ dv_g \leq C.
 \end{align*}
 Then for all positive integer $p$ we have,
 \begin{align}
  \frac{\beta_0^p}{p!} \int_{\hn} |u|^{2p} \ dv_g \leq \int_{\hn} (e^{\beta_0 u^2} - 1) \ dv_g \leq C.
 \end{align}
Therefore for all $u \in \Chn,$ we have,
\begin{align}
 \left(\int_{\hn} |u|^{2p} \ dv_g \right)^{\frac{1}{2p}} \leq \frac{C^{\frac{1}{2p}}(p!)^{\frac{1}{2p}}}{\beta_0^{\frac{1}{2}}}
 ||u||_{k,g}.
\end{align}
For general $p,$ let $n$ be the positive integer such that $n \leq p \leq n + 1.$ Then setting 
$\alpha = \frac{n(n + 1 - p)}{p},$ we have,
\begin{align} \label{lower bound on Spk}
 \left(\int_{\hn} |u|^{2p} \ dv_g \right)^{\frac{1}{2p}} &\leq
 \left(\int_{\hn} |u|^{2n} \ dv_g \right)^{\frac{\alpha}{2n}}
 \left(\int_{\hn} |u|^{2(n+1)} \ dv_g \right)^{\frac{1 - \alpha}{2(n+1)}} \notag \\
 &\leq \frac{C^{\frac{1}{2p}}(n!)^{\frac{1}{2p}}(n + 1)^{\frac{(1 - \alpha)}{2(n + 1)}}}{\beta_0^{\frac{1}{2}}}
 ||u||_{k,g}.
\end{align}
Therefore we have from \eqref{lower bound on Spk} and Stirling formula,
\begin{align*}
 2p S_{2p,k} \geq \frac{2 \beta_0 p}{C^{\frac{1}{p}} (n!)^{\frac{1}{p}}(n + 1)^{1 - \frac{n}{p}}}
   \geq 2 \beta_0 e + o(1).
\end{align*}
This gives,
\begin{align}
 \liminf_{p \rightarrow \infty}\ pS_{p,k} \geq 2 \beta_0 e.
\end{align}
To prove the opposite inequality, consider the sequence of functions
$\tilde u_m$ defined in \eqref{generalmoser}, then 
\begin{align*}
 \int_{\hn} |\tilde u_m|^{p} \ dv_g &\geq C \int_{\{|x| < \frac{1}{\sqrt{m}}\}} |v_m(x)|^p \ dx ,\\
                             &\geq C \left(\frac{\log m}{2 M}\right)^{\frac{p}{2}} \left(\frac{1}{m}\right)^{\frac{N}{2}}.   
\end{align*}
Choose $p$ such that $2k \log m = p,$ then $p$ goes to infinity as $m$ goes to infinity. We see that for
such choice of $p,$ using \eqref{generalmoser}
\begin{align}
 S_{p,k} \leq \frac{||\tilde u_m||_{k,g}^2}{\left[ \int_{\hn} |\tilde u_m|^{p} \ dv_g \right]^{\frac{2}{p}}}
 \leq \frac{2 \beta_0 e}{p} \frac{[1 + O(\frac{1}{\log m})]}{C^{\frac{2}{p}}}.
\end{align}
This gives, 
\begin{align*}
 \limsup_{p \rightarrow \infty}\  pS_{p,k} \leq 2 \beta_0 e,
\end{align*}
and the proof is complete.
\end{proof}
\subsection{Applications to Geometric PDE}
 In this section we will study a semi-linear elliptic PDE, motivated by the    $Q_{\frac{N}{2}}$-curvature problem.\\
Let $(M,g)$ be a Riemannian manifold of even dimension $N$. For integers $k< \frac{N}{2}$ we have the notion of $Q_k$ curvature given by $Q_k =\frac{2(-1)^k}{N-2k}P_k(1)$ and the notion can be extended using analytic continuation to define $Q_{\frac{N}{2}}$ curvature of the manifold (see \cite{Branson} for details). Let $\tilde g= e^{2u}g$ be a conformal metric on $(M,g)$, then the  $Q_{\frac{N}{2}}$ curvatures $Q_{\frac{N}{2},g},Q_{\frac{N}{2},\tilde{g}}$ of $g$ and $\tilde{g}$ are related by $ P_{\frac N2,g}(u) + Q_{\frac{N}{2},g} = Q_{\frac{N}{2},\tilde{g}}e^{Nu}$, where $P_{\frac N2,g}$ is the critical GJMS operator as defined in Section 3.\\\\ Motivated by this equation we investigate the following PDE in Hyperbolic space :
 \begin{align*} 
  P_k(u) + Q_1 = Q_2 e^{2u},
 \end{align*}
where $Q_1, Q_2$ are real valued functions defined on $\hn$ and $N=2k$. Note that the $Q_{\frac{N}{2}}$ curvature equation can be reduced to this equation by taking $v= \frac{N}{2}u$.\\
We prove,
\begin{thm} \label{existence result thm 1}
 Let $Q_1, Q_2 \in L^2(\hn)$ then the equation 
 \begin{align} \label{Qk curvature 2}
  P_k(u) + Q_1 = Q_2 e^{2u},
 \end{align}
  has a solution in $H^{k}(\hn) + \mathbb{R}.$
\end{thm}

The assumption of the above theorem is bit restrictive from a geometric point of view as the $Q_{\frac{N}{2}}$ curvature of $\hn$ is a constant and hence not in $L^2.$ The following theorem covers this case.
\begin{thm} \label{existence result thm 2}
 Suppose $Q_1- Q_2 \in L^2(\hn)$ and $Q_2 \le 0$ then the equation \eqref{Qk curvature 2} has a solution in $H^{k}(\hn).$
\end{thm}

Under the assumptions of the theorem, the above PDE \eqref{Qk curvature 2} has a variational structure, more precisely we may expect solutions of the above PDE as critical points of the functional
 \begin{align}
  J_Q(u) = \int_{\hn} P_k (u)u \ dv_g + 2 \int_{\hn} Q_1 u \ dv_g - \log \int_{\hn} Q_2(e^{2u} - 1) \ dv_g,
 \end{align}
in an appropriate function space. For this purpose we need a linearised form of the Adams inequality.

\begin{lem} \label{estimate on Spk2}
 Let $\delta \in (0,1)$, then there exists a constant $C(\delta) > 0$ such that the inequality 
 \begin{align}
  \log \int_{\hn} (e^u - 1)^2 \ dv_g \leq \log \int_{\hn} (e^{2u} - 2u - 1) \ dv_g \notag 
  \leq C(\delta) + \frac{1}{\beta_0 \delta} \int_{\hn} P_k (u)u \ dv_g.
 \end{align}
holds for all $u \in H^{k}(\hn)$.
\end{lem}

\begin{proof}
 Fix $\delta \in (0,1),$ then by Taylor expansion and Cauchy-Schwartz inequality we have,
 \begin{align*}
  \int_{\hn} (e^u - u - 1) \ dv_g &= \sum_{p = 2}^{\infty}\int_{\hn} \frac{1}{p!} u^p \ dv_g \\
   &\leq \sum_{p = 2}^{\infty} \frac{1}{\sqrt{p!}} \left[\frac{\int_{\hn} P_k (u)u \ dv_g}{2 \beta_0 \delta}\right]^{\frac{p}{2}}
   \frac{1}{\sqrt{p!}} \left[\frac{2 \beta_0 \delta}{S_{p,k}}\right]^{\frac{p}{2}} \\
   &\leq 
   \left[\sum_{p = 2}^{\infty} \frac{1}{p!} \left(\frac{\int_{\hn} P_k (u)u \ dv_g}{2 \beta_0 \delta}\right)^p\right]^{\frac{1}{2}}
  \left[\sum_{p = 2}^{\infty} \frac{1}{p!} \left(\frac{2 \beta_0 \delta}{S_{p,k}}\right)^p\right]^{\frac{1}{2}}.
  \end{align*}
Now by Lemma \ref{estimate on Spk} and Stirling formula we see that 
$\limsup \frac{1}{(p!)^{\frac{1}{p}}}\frac{2 \beta_0 \delta}{S_{p,k}} \leq \delta < 1.$ Hence we have,
\begin{align} \label{bound on functional 1}
 \int_{\hn} (e^u - u - 1) \ dv_g \leq c(\delta)\left[e^{\frac{\int_{\hn} P_k (u)u \ dv_g}{2 \beta_0 \delta}}
 - \frac{\int_{\hn} P_k (u)u \ dv_g}{2 \beta_0 \delta} - 1\right]^{\frac{1}{2}}.
\end{align}
Therefore applying \eqref{bound on functional 1} to $2u,$ we get,
\begin{align} \label{bound on functional 2}
 \int_{\hn} (e^{2u} - 2u - 1) \ dv_g \leq c(\delta)e^{\frac{\int_{\hn} P_k (u)u \ dv_g}{\beta_0 \delta}}.
\end{align}
Now using \eqref{bound on functional 2} and the inequality $(e^t - 1)^2 \leq (e^{2t} - 2t - 1)$ for 
all $t \in \mathbb{R},$ we get
\begin{align*}
\log \int_{\hn} (e^u - 1)^2 \ dv_g \leq \log \int_{\hn} (e^{2u} - 2u - 1) \ dv_g \notag 
  \leq C(\delta) + \frac{1}{\beta_0 \delta} \int_{\hn} P_k (u)u \ dv_g. 
\end{align*}
This completes the proof of the lemma.
\end{proof}
Proof of Theorem \ref{existence result thm 1} relies on the basic variational techniques. We need the following lemma before proceeding to the proof.
\begin{lem} \label{estimate on Spk3}
 Let $Q \in L^2(\hn),$ then the functional, $I_Q(u) = \int_{\hn} Q(e^u - 1) \ dv_g$ is uniformly
 continuous on bounded subsets of $H^{k}(\hn).$ Moreover, $I_Q$ is weakly continuous, that is,
 \begin{align*}
  u_m \rightharpoonup u \ \mbox{in} \ H^{k}(\hn) \ \ \ \ \mbox{implies} \ \ \ I_Q(u_m) \rightarrow I_Q(u).
 \end{align*}
\end{lem}
 \begin{proof}
  Let $u,v \in H^{k}(\hn)$ be such that,
  \begin{align*}
   \int_{\hn} P_k (u) u \ dv_g + \int_{\hn} P_k (v)v \ dv_g \leq C_0,
  \end{align*}
then using the inequality $(e^{t} - 1)^2 \leq |e^{2t} - 1|,$ and \eqref{bound on functional 1} we have,
\begin{align*}
 |I_Q(u) - I_Q(v)| &\leq \left(\int_{\hn} |Q|^2 \ dv_g \right)^{\frac{1}{2}}
 \left(\int_{\hn} |e^u - e^v|^2 \ dv_g \right)^{\frac{1}{2}} \\
 &= \left(\int_{\hn} |Q|^2 \ dv_g \right)^{\frac{1}{2}}
 \left(\int_{\hn} |(e^{u - v} - 1)(e^v - 1) + (e^{u - v} - 1)|^2 \ dv_g \right)^{\frac{1}{2}} \\
 &\leq C(Q) \left[\left(\int_{\hn} (e^{2v} - 1)^2\right)^{\frac{1}{4}} 
 \left(\int_{\hn} (e^{2(u - v)} - 1)^2 \ dv_g\right)^{\frac{1}{4}} \right. \\
 &\left. + \left(\int_{\hn} (e^{u - v} - 1)^2 \ dv_g\right)^{\frac{1}{2}}\right] \\
 &\leq C(Q,C_0,\delta) \left[\int_{\hn} P_k(u - v)(u - v) \ dv_g \right]^{\frac{1}{4}}.
\end{align*}
This proves the first part of the lemma.\\ To prove the second part, let $u_m \rightharpoonup u$ in $H^{k}(\hn).$
Then from $\sup_m \int_{\hn} P_k (u_m) u_m \ dv_g < \infty $ and Lemma \ref{estimate on Spk2} we see that
$\sup_m \int_{\hn} (e^{u_m} - 1)^2 \ dv_g < \infty.$ Let $\epsilon > 0$ be given, then using $Q \in L^2(\hn),$
we conclude that there exists a compact set $K$ such that,
\begin{align*}
 \left |\int_{\hn \backslash K} Q(e^{u_m} - e^{u}) \ dv_g \right| < \frac{\epsilon}{2}.
\end{align*}
Again using $\sup_m \int_{\hn} (e^{u_m} - 1)^2 \ dv_g < \infty$ and Vitali's convergence theorem we conclude
that,
\begin{align*}
 \int_{K} Q(e^{u_m} - 1) \ dv_g \rightarrow \int_{K} Q(e^{u} - 1) \ dv_g,
\end{align*}
and this completes the proof.
\end{proof}
\textbf{Proof of Theorem \ref{existence result thm 1}:} 
 Let us define $\mathcal{O} = \{u \in H^{k}(\hn) : \int_{\hn} Q_2(e^{2u} - 1) \ dv_g > 0\}.$ Then 
 $\mathcal{O}$ is an open subset of $H^{k}(\hn),$ thanks to Lemma \ref{estimate on Spk3}. Define,
 \begin{align}
  J_Q(u) = \int_{\hn} P_k (u)u \ dv_g + 2 \int_{\hn} Q_1 u \ dv_g - \log \int_{\hn} Q_2(e^{2u} - 1) \ dv_g,
 \end{align}
then $J_Q$ is well defined on $\mathcal{O}.$ We see that,
\begin{align} \label{Hardy like estimate}
 \left| \int_{\hn} Q_1 u \ dv_g \right| &\leq \left(\int_{\hn} Q^2_1 \ dv_g\right)^{\frac{1}{2}}
 \left(\int_{\hn} u^2 \ dv_g\right)^{\frac{1}{2}}, \notag \\
 &\leq c_0 \left(\int_{\hn} P_k (u)u \ dv_g\right)^{\frac{1}{2}},
\end{align}
and
\begin{align*}
 \int_{\hn} Q_2(e^{2u} - 1) \ dv_g &\leq \left( \int_{\hn} Q^2_2 \ dv_g \right)^{\frac{1}{2}}
 \left(\int_{\hn} (e^{2u} - 1)^2 \ dv_g \right)^{\frac{1}{2}}.
\end{align*}
Therefore taking logarithm and using lemma\eqref{estimate on Spk2} we get,
\begin{align} \label{other estimate}
 \log \int_{\hn} Q_2(e^{2u} - 1) \leq c_1 + \frac{2}{\beta_0 \delta} \int_{\hn} P_k (u)u \ dv_g.
\end{align}
From \eqref{Hardy like estimate} and \eqref{other estimate} we get,
\begin{align}
 J_Q(u) \geq \left(\int_{\hn} P_k (u)u \ dv_g \right)^{\frac{1}{2}}
 \left[(1 - \frac{2}{\beta_0 \delta}) \left(\int_{\hn} P_k (u)u \ dv_g \right)^{\frac{1}{2}} - c_0\right] - c_1.
\end{align}
This proves $J_Q$ is bounded from below and coercive. Let $u_m$ be sequence on $\mathcal{O}$ such that 
$J_Q(u_m) \rightarrow \inf_{u \in \mathcal{O}} J_Q(u).$ Since $J_Q$ is coercive, we can assume $u_m$ is a bounded 
sequence in $H^{k}(\hn)$ and hence $u_m \rightharpoonup u_0$ in $H^{k}(\hn).$ Clearly $u_0 \in \mathcal{O},$
otherwise $J_Q$ would become infinity, and by Lemma \ref{estimate on Spk3} we conclude that 
$J(u_0) = \inf_{u \in \mathcal{O}} J_Q(u).$

Since $\mathcal{O}$ is open, we have for all $v \in \Chn,$
\begin{align}
 \int_{\hn} P_k (u_0)v \ dv_g + \int_{\hn} Q_1 v 
 - \frac{\int_{\hn} Q_2 e^{2u_0}v}{\int_{\hn} Q_2(e^{2u_0} - 1) \ dv_g} = 0.
\end{align}
Hence, $u_0 - \frac{1}{2} \log \int_{\hn} Q_2(e^{2u_0} - 1) \ dv_g$ is a solution to the problem \eqref{Qk curvature 2}.
 \hfill $\Box$ \\\\
{\bf Proof of Theorem \ref{existence result thm 2}:}
Let us consider the following functional :
\begin{align}
 J(u) = \frac{1}{2}\int_{\hn} (P_k u)u \ dv_g - \int_{\hn} Q u \ dv_g - 
 \frac{1}{2} \int_{\hn} Q_2(e^{2u} - 2u - 1) \ dv_g,
\end{align}
where $Q = (Q_2 - Q_1),$ then $J$ is well defined on $H^k(\hn),$ and solutions of the PDE \eqref{Qk curvature 2} 
can be obtained by finding it's critical points. Since $Q_2 \leq 0$ on $\hn$ and $(e^{2t} - 2t - 1) \geq 0$ 
for all $t \in \mathbb{R},$ we have the following
coercivity estimate:
\begin{align}
 J(u) &\geq \int_{\hn} (P_k u) u \ dv_g - \left(\int_{\hn} Q^2 \ dv_g \right)^{\frac{1}{2}}\left(\int_{\hn} u^2 \ dv_g \right)^{\frac{1}{2}} \notag \\
 &\geq \int_{\hn} (P_k u) u \ dv_g - \frac{1}{\Theta}
 \left(\int_{\hn} Q^2 \ dv_g \right)^{\frac{1}{2}}\left(\int_{\hn} (P_k u) u  \ dv_g \right)^{\frac{1}{2}}.
\end{align}
Therefore $J$ is a convex and coercive functional in $H^k(\hn).$ Since 
$\int_{\hn} (-Q_2)(e^{2u} - 2u - 1) \ dv_g \geq 0,$ for all $u \in H^k(\hn),$ by Fatou's lemma $J$ is
weakly sequentially lower semicontinuous in $H^k(\hn)$. Hence by direct method in the calculus of variations,
$J$ attains its infimum in $H^k(\hn).$ Let $\tilde u \in H^k(\hn)$ be such that $J(\tilde u) = \inf J(u),$ then
one can easily check that $J(\tilde u + t v) < + \infty,$ for all $v \in \Chn,$ and therefore
\begin{align}
 0 = \frac{d}{dt} J(\tilde u + t v)|_{t = 0 } &=
 \int_{\hn} (P_k u)v \ dv_g  - \int_{\hn} Q v \ dv_g - \int_{\hn} Q_2(e^{2 \tilde u - 1})v \ dv_g \notag \\
 &= \int_{\hn} (P_k u)v \ dv_g + \int_{\hn} \left[Q_1 - Q_2 e^{\tilde u}\right] v \ dv_g.
\end{align}
This proves $\tilde u$ solves \eqref{Qk curvature 2}. 
\hfill $\Box$


\begin{thebibliography}{99}
\bibitem{A} D.R. Adams :
\emph{A sharp inequality of J. Moser for higher order derivatives,}
Ann. of Math. (2), 128 (2), 385-398 (1988).

\bibitem{AdiDruet} A. Adimurthi, O. Druet :
\emph{Blow-up analysis in dimension $2$ and a sharp form of Trudinger-Moser inequality,} 
Comm. Partial Differential Equations 29, no. 1-2, 295-322 (2004). 

\bibitem{AdiSandeep} A. Adimurthi, K. Sandeep : 
\emph{A singular Moser-Trudinger embedding and its applications,} 
NoDEA Nonlinear Differ. Equ. Appl. 13(5-6), 585-603 (2007).

\bibitem{AdiT} A. Adimurthi, K. Tinterev : 
\emph{On a version of Trudinger-Moser inequality with M\"{o}bius shift invariance,}
Calc. Var. Partial Differential Equations 39, no. 1-2, 203-212 (2010).

\bibitem{AdiYang} A. Adimurthi , Y. Yang :
\emph{An interpolation of hardy inequality and trudinger-moser inequality in $\mathbb{R}^N$ and its applications,}
International Mathematics Research Notices, vol. 13,  2394-2426 (2010).

\bibitem{ADN} S. Agmon, A. Douglis, L. Nirenberg : 
\emph{Estimates near the boundary for solutions of elliptic partial
differential equations satisfying general boundary conditions. I,} 
Comm. Pure Appl. Math. 12, 623-727 (1959).

\bibitem{BM} L. Battaglia, G. Mancini :
\emph{Remarks on the Moser-Trudinger Inequality,}
Adv. Nonlinear Anal. 2, no. 4, 389-425 (2013). 

\bibitem{Branson} T.P. Branson :
\emph{The functional determinant,} 
Global Analysis Research Center Lecture Notes Series, volume 4, Seoul National University 1993.

\bibitem{Cao} D. Cao :
\emph{Nontrivial solution of semilinear elliptic equations with critical exponent in $\rnn$,}
Communications in Partial Differential Equations, vol. 17, 407-435 (1992).

\bibitem{CZ} Y. Chang, L. Zhao :
\emph{Min-max level estimate for a singular quasilinear polyharmonic equation in $\mathbb{R}^{2m},$}
J. Differential Equations 254, no. 6, 2434-2464 (2013).

\bibitem{Ode} M. de Souza, J.M. do \'{O} :
\emph{On a class of singular Trudinger-Moser type inequalities and its applications,}
Math. Nachr. 284 (14-15), 1754-1776 (2011).

\bibitem{DoO} J.M. do \'{O} :
\emph{N-Laplacian equations in $\rn$ with critical growth,}
Abstract and Applied Analysis, vol. 2, pp. 301-315 (1997).

\bibitem{Fonta} L. Fontana :
\emph{Sharp borderline Sobolev inequalities on compact Riemannian manifolds,}
Comment. Math. Helv., 68, no. 3, 415-454 (1993). 

\bibitem{FoM} L. Fontana, C. Morpurgo :\emph{Sharp Adams and Moser-Trudinger inequalities on $\mathbb{R}^n$ and other spaces of infinite measure},Preprint, arXiv:1504.04678 [math.AP]

\bibitem{GJMS} C.R. Graham, R. Jenne, L.J. Mason, G.A.J. Sparling :
\emph{Conformally invariant powers of the Laplacian. I. Existence,} 
Journal of the London Mathematical Society, (2) 46, 557-565 (1992).

\bibitem{juhl} A. Juhl :
\emph{Explicit formulas for GJMS-operators and Q-curvatures,} 
Geom. Funct. Anal. 23, no. 4, 1278-1370 (2013).

\bibitem{LiRuf} Y. Li, B. Ruf :
\emph{A sharp Trudinger-Moser type inequality for unbounded domains in $\mathbb{\rn}$,}
Indiana University Mathematics Journal, vol. 57, no. 1, 451-480 (2008).

\bibitem{Liu} G. Liu :
\emph{Sharp higher-order Sobolev inequalities in the hyperbolic space $\mathbb{\hn}$,}
Calc. Var. Partial Differential Equations 47, no. 3-4, 567-588 (2013).

\bibitem{LuT} G. Lu, H. Tang :
\emph{Best constants for Moser-Trudinger inequalities on high dimensional hyperbolic spaces,}
Adv. Nonlinear Stud. 13, no. 4, 1035-1052 (2013).

\bibitem{MS} G. Mancini, K. Sandeep :
\emph{Moser-Trudinger inequality on conformal discs, }
Communications in Contemporary Mathematics, vol. 12, no. 6, 1055-1068 (2010).

\bibitem{MST}G. Mancini, K. Sandeep, C. Tintarev :
\emph{Trudinger-Moser inequality in the hyperbolic space $\hn$,}
Adv. Nonlinear Anal. 2, no. 3, 309-324 (2013).

\bibitem{Moser} J. Moser :
\emph{A sharp form of an inequality by N. Trudinger,}
Indiana Univ. Math. J. 20, 1077-1092 (1971). 

\bibitem{Owen} M.P. Owen :
\emph{The Hardy-Rellich inequality for polyharmonic operators,} 
Proc. Roy. Soc. Edinburgh Sect. A 129, 825-839 (1999).

\bibitem{Panda} R. Panda :
\emph{Nontrivial solution of a quasilinear elliptic equation with critical growth in $\mathbb{\rn}$,}
Proceedings of the Indian Academy of Science, vol. 105, pp. 425-444 (1995).

\bibitem{Pohozaev} S.I. Poho\v{z}aev :
\emph{The Sobolev imbedding in the case $pl = n$,} 
In: Proc. Tech. Sci.Conf. on Adv. Sci. Research 1964-1965, Mathematics Section, Moskov. Energet.Inst., Moscow, 158-170 (1965).

\bibitem{Ratcliffe} J.G. Ratcliffe :
\emph{Foundations of hyperbolic manifolds (Second edition), }
Graduate Texts in Mathematics, 149. Springer, New York (2006).

\bibitem{Ruf} B. Ruf :
\emph{A sharp Trudinger-Moser type inequality for unbounded domains in $\mathbb{R}^2$,}
J. Funct. Anal. 219, no. 2, pp. 340-367 (2005).

\bibitem{tar} C. Tarsi : 
\emph{Adams' inequality and limiting Sobolev embeddings into Zygmund spaces,} 
Potential Anal. 37, no. 4, 353-385 (2012).

\bibitem{Tin} C. Tintarev : 
\emph{Trudinger-Moser inequality with remainder terms,} 
J. Funct. Anal. 266, no. 1, 55-66 (2014).

\bibitem{Trudinger} N.S. Trudinger :
\emph{On imbeddings into Orlicz spaces and some applications,} 
J. Math. Mech. 17, 473-483 (1967).

\bibitem{Wolf} J. Wolf : 
\emph{Spaces of Constant Curvature,} 
McGraw-Hill, New York (1967)

\bibitem{Yang} Y. Yang : 
\emph{A sharp form of Moser-Trudinger inequality in high dimension,}
J. Funct. Anal. 239, no. 1, 100-126 (2006). 

\bibitem{YZ} Y. Yang, X. Zhu :
\emph{Trudinger-Moser embedding on the hyperbolic space,}
Abstr. Appl. Anal., Art. ID 908216, 4, 46E35 (35A23) (2014).

\end{thebibliography}
\end{document}